\documentclass[reqno]{amsart}
\usepackage{amssymb}
\usepackage{hyperref}
\usepackage[english]{babel}

\numberwithin{equation}{section}
\newtheorem{theorem}{Theorem}[section]

\newtheorem{proposition}[theorem]{Proposition}
\newtheorem{corollary}[theorem]{Corollary}

\begin{document}
\title[On fractional powers of Bessel operators]
      {On fractional powers of Bessel operators}

\author[ E.L. Shishkina, S.M. Sitnik]
{Elina L. Shishkina, Sergei M. Sitnik 
 \\   \\  \\  \vspace*{1.3cm}
\hfill{\scriptsize
{\it Dedicated to Professor Ivan Dimovski's contributions}}\\
\vspace*{-1.5cm}}

\address{E.L. Shishkina \newline
Voronezh State University, Universitetskaya Pl. 1, Voronezh, 394000, Russia}
\email{ilina\_dico@mail.ru}

\address{S.M. Sitnik \newline
Voronezh Institute of the Ministry of Internal Affairs, Pr. Patriotov, 53 \hfill \break
Voronezh, 394065, Russia \newline
\textit{and}\newline
 RUDN 
 University,  6 Miklukho--Maklaya Str. \hfill \break
  Moscow, 117198, Russia}
\email{pochtasms@gmail.com}


\begin{abstract}

 In this paper we study fractional powers of the Bessel differential operator. The fractional powers are defined explicitly in the integral form without use of integral transforms in its definitions. Some general properties of the fractional powers of the Bessel differential operator are proved and some are listed. Among them are different variations of definitions, relations with the Mellin and Hankel transforms, group property, generalized Taylor formula with Bessel operators, evaluation of resolvent integral operator in terms of the Wright or generalized Mittag--Leffler functions. At the end, some topics are indicated for further study and possible generalizations. Also the aim of the paper is to attract attention and give references to not widely known results on fractional powers of the Bessel differential operator.

\end{abstract}

\maketitle

\vspace*{-0.8cm}
\section{Introduction and historical remarks}

\textit{This paper was published in \cite{SiShi} in the special issue of the Journal of Inequalities and Special Functions
dedicated to  Professor Ivan Dimovski's contributions to different fields of mathematics: transmutation theory, special functions, integral transforms, function theory etc.}

In this paper we study the  differential Bessel operator in the form
\begin{equation}\label{Bess}
B_\nu= D^2+\frac{\nu}{x}D,\qquad \nu\geq 0, \ \ \ D:= \frac{d}{dx},
\end{equation}
and its fractional powers $(B_\nu)^{\alpha}, \alpha\in \mathbb{R}$.
This operator has essential role in the theory of differential equations both as a radial part of the Laplace operator and also as involved in partial differential equations with Bessel operators. Such equations were called $B$--elliptic, $B$--hyperbolic and $B$--parabolic by I.A.~Kipriyanov and intensively studied by his scientific school and many others researchers, now the term ``Laplace--Bessel equations" is also used. For equations with Bessel operators and related topics cf. \cite{CSh},\cite{Kip},\cite{Mat}.

Of course fractional powers of the Bessel operator \eqref{Bess} were studied in many papers. But in the most of them fractional powers were defined implicitly as a power function multiplication under Hankel transform. This definition via integral transforms leads to many restrictions. Just imagine that for  the classical Riemann--Liouville fractional integrals we have to work only with its definitions via Laplace or Mellin transforms and nothing more  without explicit integral representations. If it would be true, then 99\% of classical ``Bible" \cite{KK} and other books on fractional calculus would be empty as they mostly use explicit integral definitions! But for fractional powers of the Bessel operator at most papers implicit definitions via Hankel transform are  still used.

Of course such situation is not natural and in some papers different approaches to step closer to explicit formulas were studied. Let us mention that in \cite{McB} explicit formulas were derived as compositions of Erd\'{e}lyi--Kober fractional integrals \cite{KK} on distribution spaces, in this monograph results on fractional powers of Bessel and related operators are gathered of McBride's and earlier papers.
An important step was done in  \cite{Ida} in which explicit definitions  were derived in terms of the Gauss hypergeometric functions with different applications to PDE, we also use basic formulas from \cite{Ida} in this paper.
The most general study was fulfilled by I.~Dimovski  
and V.~Kiryakova \cite{Dim66}, \cite{Dim68}, \cite{DimKir}, \cite{Kir1} for the more general class of hyper--Bessel differential operators
related to the Obrechkoff integral transform.
They constructed explicit integral representations of the fractional powers of these operators by using Meijer $G$--functions as kernels, and also intensively and successfully used for this the theory of transmutations.
Note that in this and others fields of  theoretical and applied mathematics, the methods of transmutation theory are very useful and productive and for some problems are even irreplaceable (see e.g. \cite{Dim}).
In \cite{Sita1}--\cite{Sita2}  simplified representations for fractional powers of the Bessel operator were derived with Legendre functions as kernels, and based on them general definitions were simplified and unified with standard fractional calculus notation as in \cite{KK}, and also important generalized Taylor formulas were proved which mix integer powers of Bessel operators (instead of derivatives in the classical Taylor formula) with fractional power of the Bessel operator as integral remainder term, cf. also \cite{Sita3}--\cite{Sita4}.

\smallskip

In this paper we study fractional powers of the Bessel differential operator. The fractional powers are defined explicitly in the integral form without use of integral transforms in its definitions. Some general properties of the fractional powers of the Bessel differential operator are proved and some are listed. Among them are different variations of definitions, relations with the Mellin and Hankel transforms, group property, generalized Taylor formula with Bessel operators, evaluation of resolvent integral operator in terms of the Wright or generalized Mittag--Leffler functions. At the end, some topics are indicated for further study and possible generalizations. Also the aim of the paper is to attract attention and give references to not widely known results on fractional powers of the Bessel differential operator. Due to this, the paper contains more prolonged and detailed reference list.

\vspace*{-0.45cm}

\section{Definitions}

We give definitions and notations for fractional powers of the Bessel operator \eqref{Bess} following \cite{Sita1}--\cite{Sita2} in the unified way with ``the Bible on Fractional Calculus" \cite{KK}.

\smallskip

\noindent
\textbf{Definition 1.}
Let $\alpha>0$.  The right--sided fractional Bessel integral $B_{\nu,b-}^{-\alpha}$ for $f(x)\in C^{[2\alpha]+1}(0,b]$, $b\in (0,+\infty)$, is defined by the formula
\vskip -13pt
\begin{equation}\label{Bess1}
(B_{\nu,b-}^{-\alpha}f)(x){=}\frac{1}{\Gamma(2\alpha)}\int\limits_x^b\left(\frac{y^2{-}x^2}{2y}\right)^{2\alpha-1}\,_2F_1\left(\alpha{+}\frac{\nu{-}1}{2},\alpha;2\alpha;1{-}\frac{x^2}{y^2}\right)f(y)dy.
\end{equation}
\vskip -3pt \noindent
The left--sided fractional Bessel integral $B_{\nu,a+}^{-\alpha}$  for $f(x)\in C^{[2\alpha]+1}[a,+\infty)$, $a{\in}(0,+\infty)$, is defined by the formula
\vskip -10pt
\begin{equation}\label{Bess2}
(B_{\nu,a+}^{-\alpha}f)(x){=}\frac{1}{\Gamma(2\alpha)}\int\limits_a^x\left(\frac{y}{x}\right)^\nu\left(\frac{x^2{-}y^2}{2x}\right)^{2\alpha-1}\,_2F_1\left(\alpha{+}\frac{\nu{-}1}{2},\alpha;2\alpha;1{-}\frac{y^2}{x^2}\right)f(y)dy.
\end{equation}

The above Definition 1 is based on integral representations introduced for special cases $a=1, b=1$  in \cite{Ida}.

The cases $b{=}+\infty$ and $a{=}0$ were not studied before as they require more delicate considerations and estimates when applied. But they seem to be very important as in most applications boundary conditions for differential equations are given exactly at zero or infinity. So we introduce two more fractional Bessel integrals for these special values.

\smallskip

\noindent
\textbf{Definition 2.}
Let $\alpha>0$.  The right--sided fractional Bessel integral $B_{\nu,-}^{-\alpha}$ for $f(x){\in}C^{[2\alpha]+1}(0,+\infty)$  is defined by the formula
\vskip -11pt
\begin{equation}\label{Bess3}
(B_{\nu,-}^{-\alpha}f)(x){=}\frac{1}{\Gamma(2\alpha)}\int\limits_x^\infty\left(\frac{y^2{-}x^2}{2y}\right)^{2\alpha-1}\,_2F_1\left(\alpha{+}\frac{\nu{-}1}{2},\alpha;2\alpha;1{-}\frac{x^2}{y^2}\right)f(y)dy.
\end{equation}
\vskip -3pt \noindent
The left--sided fractional Bessel integral $B_{\nu,0+}^{-\alpha}$  for $f(x){\in}C^{[2\alpha]+1}[0,+\infty)$ is defined by the formula
\vskip -13pt
\begin{equation}\label{Bess4}
(B_{\nu,0+}^{-\alpha}f)(x){=}\frac{1}{\Gamma(2\alpha)}\int\limits_0^x\left(\frac{y}{x}\right)^\nu\left(\frac{x^2{-}y^2}{2x}\right)^{2\alpha-1}\,_2F_1\left(\alpha{+}\frac{\nu{-}1}{2},\alpha;2\alpha;1{-}\frac{y^2}{x^2}\right)f(y)dy.
\end{equation}

It was noted in \cite{Sita1}--\cite{Sita2} (cf. also \cite{Sita4}) that Definitions 1 and 2 may be simplified, as the kernels are expressed in the more simple way via Legendre functions (the Legendre functions are two--parameter family but the Gauss hypergeometric functions are in general three--parameter family). This simplification is based on the next formula  from \cite{IR3},
\vskip -12pt
$$
_2F_1(a,b;2b;z)=2^{2b-1}\Gamma\left(b+\frac{1}{2}\right)\, z^{\frac{1}{2}-b}(1-z)^{\frac{1}{2}\left(b-a-\frac{1}{2}\right)}
P_{a-b-\frac{1}{2}}^{\frac{1}{2}-b}\left[\left(1-\frac{z}{2}\right)\frac{1}{\sqrt{1-z}}\right],
$$
and leads to the next simplified definitions on proper functions as above.

\smallskip

\noindent
\textbf{Definition 3.}
Let $\alpha>0$.
The right--sided fractional Bessel integral $B_{\nu,b-}^{-\alpha}$ for $f(x)\in C^{[2\alpha]+1}(0,b]$, $b\in (0,+\infty)$,
is defined by the formula
\vskip - 11pt
$$
(B_{\nu,b-}^{-\alpha}f)(x)=\frac{\sqrt{\pi}}{2^{2\alpha-1}\Gamma(\alpha)}\int\limits_x^b(y^2-x^2)^{\alpha-\frac{1}{2}}\,
\left(\frac{y}{x}\right)^{\frac{\nu}{2}}P_{\frac{\nu}{2}-1}^{\frac{1}{2}-\alpha}
\left[\frac{1}{2}\left(\frac{x}{y}+\frac{y}{x}\right)\right]f(y)dy,
$$
\vskip -3pt \noindent
and the left--sided fractional Bessel integral $B_{\nu,a+}^{-\alpha}$ for $f(x)\in C^{[2\alpha]+1}[a,+\infty)$, $a{\in}(0,+\infty)$,  is defined by the formula
\vskip -10pt
$$
(B_{\nu,a+}^{-\alpha}f)(x)=\frac{\sqrt{\pi}}{2^{2\alpha-1}\Gamma(\alpha)}
\int\limits_a^x(x^2-y^2)^{\alpha-\frac{1}{2}}\left(\frac{y}{x}
\right)^{\frac{\nu}{2}}P_{\frac{\nu}{2}-1}^{\frac{1}{2}-\alpha}\left[\frac{1}{2}
\left(\frac{x}{y}+\frac{y}{x}\right)\right]f(y)dy.
$$

Some important special cases are covered also by the following definition.

\smallskip

\noindent
\textbf{Definition 4.}
Let $\alpha>0$.
The right--sided fractional Bessel integral $B_{\nu,-}^{-\alpha}$ $f(x) $ ${\in}C^{[2\alpha]+1}(0,+\infty)$ is defined by the formula
\vskip -10pt
$$
(B_{\nu,-}^{-\alpha}f)(x)=\frac{\sqrt{\pi}}{2^{2\alpha-1}\Gamma(\alpha)}\int\limits_x^\infty(y^2-x^2)^{\alpha-\frac{1}{2}}\,
\left(\frac{y}{x}\right)^{\frac{\nu}{2}}P_{\frac{\nu}{2}-1}^{\frac{1}{2}-\alpha}
\left[\frac{1}{2}\left(\frac{x}{y}+\frac{y}{x}\right)\right]f(y)dy,
$$
\vskip -3pt \noindent
and the left--sided fractional Bessel integral $B_{\nu,0+}^{-\alpha}$ $f(x){\in}C^{[2\alpha]+1}(0,+\infty)$  is defined by the formula
\vskip -13pt
$$
(B_{\nu,0+}^{-\alpha}f)(x)=\frac{\sqrt{\pi}}{2^{2\alpha-1}\Gamma(\alpha)}
\int\limits_0^x(x^2-y^2)^{\alpha-\frac{1}{2}}\left(\frac{y}{x}
\right)^{\frac{\nu}{2}}P_{\frac{\nu}{2}-1}^{\frac{1}{2}-\alpha}\left[\frac{1}{2}
\left(\frac{x}{y}+\frac{y}{x}\right)\right]f(y)dy.
$$

To define positive powers compositions with natural powers of the Bessel operator \eqref{Bess},
 we use the standard way and analytical continuation  for complex powers just as in \cite{KK}.
So we have a complete set of definitions for fractional powers of the Bessel differential operator \eqref{Bess}.

\vspace*{-8pt}
\section{Testing definitions}

 To ensure a reader that Definitions 1--4 really define fractional powers of the Bessel differential operator \eqref{Bess},
 let us made simple calculations for two cases.

\medskip

\textbf{1.} As we can see from \eqref{Bess} when $\nu=0$ we obtain that Bessel operator is equal to the second derivative. It is easy to check that for $\nu=0$ the next equalities are true:
\vskip -12pt
\begin{equation}\label{Rl1}
    (B_{0,b-}^{-\alpha}f)(x)=\frac{1}{\Gamma(2\alpha)}\int\limits_x^b(y-x)^{2\alpha-1}f(y)dy=(I_{b-}^{2\alpha}f)(x),
\end{equation}
\begin{equation}\label{Rl2}
    (B_{0,a+}^{-\alpha}f)(x)=\frac{1}{\Gamma(2\alpha)}\int\limits_a^x(x-y)^{2\alpha-1}f(y)dy=(I_{a+}^{2\alpha}f)(x),
\end{equation}
\begin{equation}\label{l1}
    (B_{0,-}^{-\alpha}\,f)(x)=\frac{1}{\Gamma(2\alpha)}\int\limits_x^{+\infty}(y-x)^{2\alpha-1}f(y)dy=(I_{-}^{2\alpha}f)(x),
\end{equation}
 where $I_{b-}^{2\alpha}$ and $I_{a+}^{2\alpha}$ are the right-sided and left-sided  Riemann-Liouville fractional integrals respectively (see formulas 2.17, 2.18 on p. 33 in \cite{KK}) and $I_{-}^{2\alpha}$ is the Liouville fractional integral (see formula 5.3 on p. 94 in \cite{KK}).

 \begin{proof}

\begin{enumerate}

  \item First we consider
  \vskip -10pt
  $$(B_{0,b-}^{-\alpha}f)(x)=\frac{1}{\Gamma(2\alpha)}\int\limits_x^b\left(\frac{y^2-x^2}{2y}\right)^{2\alpha-1}
  \,_2F_1\left(\alpha-\frac{1}{2},\alpha;2\alpha;1-\frac{x^2}{y^2}\right)f(y)dy.
 $$
Putting $\alpha-\frac{1}{2}=p$ in $\,_2F_1\left(\alpha-\frac{1}{2},\alpha;2\alpha;1-\frac{x^2}{y^2}\right)$,  we obtain
$$
\,_2F_1\left(\alpha-\frac{1}{2},\alpha;2\alpha;1-\frac{x^2}{y^2}\right)=\,_2F_1\left(p,p+\frac{1}{2};2p+1;1-\frac{x^2}{y^2}\right).
$$
Using formula from \cite{KK}
$$
\,_2F_1\left(p,p+\frac{1}{2};2p+1;z\right)=2^{2p}[1+(1-z)^{\frac{1}{2}}]^{-2p}
$$
we derive
\vskip -12pt
$$
\,_2F_1\left(p,p+\frac{1}{2};2p+1;1-\frac{x^2}{y^2}\right)=2^{2p}\left[1+\frac{x}{y}\right]^{-2p},
$$
\vskip -3pt \noindent
or
\vskip -10pt
$$
\,_2F_1\left(\alpha-\frac{1}{2},\alpha;2\alpha;1-\frac{x^2}{y^2}\right)=2^{2\alpha-1}\left[1+\frac{x}{y}\right]^{1-2\alpha}=\left[\frac{2y}{x+y}\right]^{2\alpha-1}.
$$
Then
\vskip -10pt
$$
\left(\frac{y^2-x^2}{2y}\right)^{2\alpha-1}\,_2F_1\left(\alpha-\frac{1}{2},\alpha;2\alpha;1-\frac{x^2}{y^2}\right)=(y-x)^{2\alpha-1}
$$
and
\vskip -12pt
$$
(B_{0,b-}^{-\alpha}f)(x)=\frac{1}{\Gamma(2\alpha)}\int\limits_x^b(y-x)^{2\alpha-1}f(y)dy=(I_{b-}^{2\alpha}f)(x).
$$

\item  Now we can put $\nu=0$ for the second fractional Bessel operator:
$$
(B_{0,a+}^{-\alpha}f)(x)=\frac{1}{\Gamma(2\alpha)}\int\limits_a^x\left(\frac{x^2-y^2}{2x}\right)^{2\alpha-1}
\,_2F_1\left(\alpha-\frac{1}{2},\alpha;2\alpha;1-\frac{y^2}{x^2}\right)f(y)dy
$$
\vspace*{-6pt}
$$
=\frac{1}{\Gamma(2\alpha)}\int\limits_a^x\left(\frac{x^2-y^2}{2x}\right)^{2\alpha-1}\left[\frac{2x}{x+y}\right]^{2\alpha-1}f(y)dy
$$
\vspace*{-6pt}
$$
=\frac{1}{\Gamma(2\alpha)}\int\limits_a^x (x-y)^{2\alpha-1} f(y)dy=(I_{a+}^{2\alpha}f)(x).
$$

 \item   Finally we prove \eqref{l1}. It follows from
 $$
(B_{0,-}^{-\alpha}\,f)(x)=\frac{1}{\Gamma(2\alpha)}\int\limits_x^{+\infty}\left(\frac{y^2-x^2}{2y}\right)^{2\alpha-1}
\,_2F_1\left(\alpha-\frac{1}{2},\alpha;2\alpha;1-\frac{x^2}{y^2}\right)f(y)dy
 $$
 \vspace*{-6pt}
 $$
 =\frac{1}{\Gamma(2\alpha)}\int\limits_x^{+\infty} (y-x)^{2\alpha-1}f(y)dy=(I_{-}^{2\alpha}f)(x).
 $$
\end{enumerate}
 \end{proof}

\textbf{2.} Let us check the following properties of the fractional Bessel integrals.

\smallskip

When $\alpha{=}1$ and $\lim\limits_{x\rightarrow b-0}g(x){=}0$, $\lim\limits_{x\rightarrow b-0}g'(x){=}0$ the right--sided fractional Bessel integral is left inverse to the differential Bessel operator
$$
(B_{\nu,b-}^{-1}B_\nu g(x))(x)=g(x).
$$

When $\alpha{=}1$ and $\lim\limits_{x\rightarrow a+0}g(x){=}0$, $\lim\limits_{x\rightarrow a+0}g'(x){=}0$ the left--sided fractional Bessel integral is left inverse to the differential Bessel operator
$$
(B_{\nu,a+}^{-1}B_\nu g(x))(x)=g(x).
$$

When $\alpha{=}1$ and $\lim\limits_{x\rightarrow +\infty}g(x){=}0$, $\lim\limits_{x\rightarrow +\infty}g'(x){=}0$ the left--sided fractional Bessel integral $B_{\nu,-}^{-1}$ is left inverse to the differential Bessel operator
$$
(B_{\nu,-}^{-1}B_\nu g(x))(x)=g(x).
$$

\begin{proof} 

\begin{enumerate}
  \item Let $\alpha{=}1$ in $B_{\nu,b-}^{-\alpha}$, then we have
  \vskip - 10pt
$$
(B_{\nu,b-}^{-1}f)(x)=\int\limits_x^b\left(\frac{y^2-x^2}{2y}\right)\,_2F_1\left(\frac{\nu+1}{2},1;2;1-\frac{x^2}{y^2}\right)f(y)dy.
$$
For $\alpha{=}1$  we can transform the hypergeometric function:
$$
\,_2F_1\left(\frac{\nu+1}{2},1;2;1-\frac{x^2}{y^2}\right)=\frac{\Gamma(2)}{\Gamma(1)\Gamma(1)}\int\limits_0^1 \left(1+\left(\frac{x^2}{y^2}-1\right)t\right)^{-\frac{\nu+1}{2}}dt
$$
\vspace*{-6pt}
$$
=\frac{1}{\frac{x^2}{y^2}-1}\,\frac{2}{1-\nu}\left(1+\left(\frac{x^2}{y^2}-1\right)t\right)^{1-\frac{\nu+1}{2}}\biggr|_{t=0}^{t=1}
$$
\vspace*{-6pt}
$$
=\frac{y^2}{x^2-y^2}\,\frac{2}{1-\nu}\left[\left(1+\left(\frac{x^2}{y^2}-1\right)\right)^{\frac{1-\nu}{2}}-1\right]
=\frac{2}{1-\nu}\,\frac{y^2}{x^2-y^2}\left[\left(\frac{x}{y}\right)^{1-\nu}-1\right].
$$
We obtain the formula
\begin{equation}\label{HypGeom}
\,_2F_1\left(\frac{\nu+1}{2},1;2;1-\frac{x^2}{y^2}\right)=\frac{2}{1-\nu}\,\frac{y^2}{x^2-y^2}\left[\left(\frac{x}{y}\right)^{1-\nu}-1\right].
\end{equation}
Then,
\vskip -11pt
$$
(B_{\nu,b-}^{-1}f)(x)=\frac{2}{1-\nu}\int\limits_x^b\left(\frac{y^2-x^2}{2y}\right)\, \frac{y^2}{x^2-y^2}\left[\left(\frac{x}{y}\right)^{1-\nu}-1\right] f(y)dy
$$
\vspace*{-6pt}
$$
=\frac{1}{\nu-1}\int\limits_x^b\, y\,\left[\left(\frac{x}{y}\right)^{1-\nu}-1\right] f(y)dy.
$$
Let $f(x)=B_\nu g(x)=g''(x)+\frac{\nu}{x}g'(x)$ and $\lim\limits_{x\rightarrow b-0}g(x)=0$, $\lim\limits_{x\rightarrow b-0}g'(x)=0$,
  then
  \vskip -12pt
$$
(B_{\nu,b-}^{-1}f)(x)=(B_{\nu,b-}^{-1}B_\nu g)(x)=\frac{1}{\nu-1}\int\limits_x^b y\left[\left(\frac{x}{y}\right)^{1-\nu}-1\right]\left(g''(y)+\frac{\nu}{y}g'(y)\right)dy
$$
\vspace*{-6pt}
$$
=\frac{1}{\nu-1}\left[\int\limits_x^b y\left[\left(\frac{x}{y}\right)^{1-\nu}-1\right]g''(y)dy+\nu\int\limits_x^b\left[\left(\frac{x}{y}\right)^{1-\nu}-1\right]g'(y)dy\right].
$$
Integrating by parts the first term twice we obtain
$$
\int\limits_x^b y\left[\left(\frac{x}{y}\right)^{1-\nu}{-}1\right]g''(y)dy{=}y\left[\left(\frac{x}{y}\right)^{1-\nu}{-}1\right]g(y)\biggr|_{y=x}^{y=b}
{-}\int\limits_x^b(\nu x^{1-\nu}y^{\nu-1}{-}1)g'(y)dy
$$
\vspace*{-6pt}
$$
=-(\nu x^{1-\nu}y^{\nu-1}-1)g(y)\biggr|_{y=x}^{y=b}+\nu(\nu-1)x^{1-\nu}\int\limits_x^b y^{\nu-2}g(y)dy
$$
\vspace*{-6pt}
$$
=(\nu-1)g(x)+\nu(\nu-1)x^{1-\nu}\int\limits_x^b y^{\nu-2}g(y)dy.
$$
Integrating by parts the second term once, we obtain
$$
\int\limits_x^b\left[\left(\frac{x}{y}\right)^{1-\nu}-1\right]g'(y)dy=\left[\left(\frac{x}{y}\right)^{1-\nu}-1\right]g(y)
\biggr|_{y=x}^{y=b}-\frac{\nu-1}{x^{\nu-1}}\int\limits_x^b y^{\nu-2}g(y)dy
$$
\vspace*{-6pt}
$$
=-(\nu-1)x^{1-\nu}\int\limits_x^b y^{\nu-2}g(y)dy.
$$

Hence, \vskip -12pt
$$
(B_{\nu,b-}^{-1}B_\nu g)(x)=\frac{1}{(\nu-1)}\biggl[(\nu-1)g(x)+\nu(\nu-1)x^{1-\nu}\int\limits_x^b y^{\nu-2}g(y)dy$$
\vspace*{-6pt}
$$ - \, \nu(\nu-1)x^{1-\nu}\int\limits_x^b y^{\nu-2}g(y)dy\biggr]=g(x).
$$

  \item Let $\alpha=1$ in $B_{\nu,a+}^{-\alpha}$, then we have
  $$
  (B_{\nu,a+}^{-1}f)(x)=\int\limits_a^x\left(\frac{y}{x}\right)^\nu\left(\frac{x^2-y^2}{2x}\right)\,_2F_1\left(\frac{\nu+1}{2},1;2;1-\frac{y^2}{x^2}\right)f(y)dy.
  $$
 Using \eqref{HypGeom}, we get
 \vskip -10pt
  $$
  (B_{\nu,a+}^{-1}f)(x)=\frac{2}{1-\nu}\int\limits_a^x\left(\frac{y}{x}\right)^\nu\left(\frac{x^2-y^2}{2x}\right)\frac{x^2}{y^2-x^2}
  \left[\left(\frac{y}{x}\right)^{1-\nu}-1\right]f(y)dy
  $$
  \vspace*{-6pt}
 $$
=\frac{x}{\nu-1}\int\limits_a^x\left(\frac{y}{x}\right)^\nu\left[\left(\frac{y}{x}\right)^{1-\nu}-1\right]f(y)dy=\frac{1}{\nu-1}\int\limits_a^x y\left[1-\left(\frac{y}{x}\right)^{\nu-1}\right]f(y)dy.
  $$
Let $f(x){=}B_\nu g(x){=}g''(x){+}\frac{\nu}{x}g'(x)$ and $\lim\limits_{x\rightarrow a+0}g(x){=}0$, $\lim\limits_{x\rightarrow a+0}g'(x){=}0$,
  then
$$
(B_{\nu,a+}^{-1}f)(x)=(B_{\nu,a+}^{-1}B_\nu g)(x)=\frac{1}{\nu-1}\int\limits_a^x y\left[1-\left(\frac{y}{x}\right)^{\nu-1}\right]\left(g''(y)+\frac{\nu}{y}g'(y)\right)dy
$$
\vspace*{-6pt}
$$
=\frac{1}{\nu-1}\left[\int\limits_a^x y\left[1-\left(\frac{y}{x}\right)^{\nu-1}\right]g''(y)dy+\nu\int\limits_a^x \left[1-\left(\frac{y}{x}\right)^{\nu-1}\right]g'(y)dy\right].
$$
Integrating by parts the first term twice we obtain
\vskip -10pt
 $$
 \int\limits_a^x y\left[1-\left(\frac{y}{x}\right)^{\nu-1}\right]g''(y)dy$$
 \vspace*{-6pt}
 $$=y\left[1-\left(\frac{y}{x}\right)^{\nu-1}\right]g''(y)\biggr|_{y=a}^{y=x}-\int\limits_a^x \left[1-\nu\left(\frac{y}{x}\right)^{\nu-1}\right]g'(y)dy
 $$
 \vspace*{-6pt}
 $$
 =-\left[1-\nu\left(\frac{y}{x}\right)^{\nu-1}\right]g(y)\biggr|_{y=a}^{y=x}-\nu(\nu-1)x^{1-\nu}\int\limits_a^x y^{\nu-2}g(y)dy
 $$
 \vspace*{-6pt}
 $$
 =(\nu-1)g(x)-\nu(\nu-1)x^{1-\nu}\int\limits_a^x y^{\nu-2}g(y)dy.
 $$
Integrating by parts the second term once we obtain
 $$
 \int\limits_a^x \left[1-\left(\frac{y}{x}\right)^{\nu-1}\right]g'(y)dy=\left[1-\left(\frac{y}{x}\right)^{\nu-1}\right]g(y)\biggr|_{y=a}^{y=x}+(\nu-1)x^{1-\nu} \int\limits_a^x y^{\nu-2}g(y)dy
 $$
 \vspace*{-6pt}
 $$
 =(\nu-1)x^{1-\nu} \int\limits_a^x y^{\nu-2}g(y)dy.
 $$
Returning to $(B_{\nu,a+}^{-1}B_\nu g)(x)$, we have
 $$
(B_{\nu,a+}^{-1}B_\nu g)(x)=g(x)-\nu x^{1-\nu}\int\limits_a^x y^{\nu-2}g(y)dy+\nu x^{1-\nu} \int\limits_a^x y^{\nu-2}g(y)dy=g(x).
$$

  \item Let $\alpha=1$ in $B_{\nu,-}^{-\alpha}$, then we have
  \vskip -10pt
$$
(B_{\nu,-}^{-1}f)(x)=\int\limits_x^{+\infty}\left(\frac{y^2-x^2}{2y}\right)\,_2F_1\left(\frac{\nu+1}{2},1;2;1-\frac{x^2}{y^2}\right)f(y)dy.
$$
Then using \eqref{HypGeom} we have
$$
(B_{\nu,-}^{-1}f)(x)=\frac{1}{\nu-1}\int\limits_x^{+\infty}\, y\,\left[\left(\frac{x}{y}\right)^{1-\nu}-1\right] f(y)dy.
$$
Let $f(x)=B_\nu g(x)=g''(x)+\frac{\nu}{x}g'(x)$ and $\lim\limits_{x\rightarrow +\infty}g(x)=0$, $\lim\limits_{x\rightarrow +\infty}g'(x)=0$,
  then
  \vskip -10pt
$$
(B_{\nu,-}^{-1}f)(x)=(B_{\nu,-}^{-1}B_\nu g)(x)=\frac{1}{\nu-1}\int\limits_x^{+\infty} y\left[\left(\frac{x}{y}\right)^{1-\nu}-1\right]\left(g''(y)+\frac{\nu}{y}g'(y)\right)dy
$$
\vspace*{-6pt}
$$
=\frac{1}{\nu-1}\left[\int\limits_x^{+\infty} y\left[\left(\frac{x}{y}\right)^{1-\nu}-1\right]g''(y)dy+\nu\int\limits_x^{+\infty}\left[\left(\frac{x}{y}\right)^{1-\nu}-1\right]g'(y)dy\right].
$$
Integrating by parts the first term twice we obtain
$$
\int\limits_x^{+\infty} y\left[\left(\frac{x}{y}\right)^{1-\nu}-1\right]g''(y)dy$$
\vspace*{-6pt}
$$=y\left[\left(\frac{x}{y}\right)^{1-\nu}-1\right]g(y)\biggr|_{y=x}^{y={+\infty}}
-\int\limits_x^{+\infty}(\nu x^{1-\nu}y^{\nu-1}-1)g'(y)dy
$$
\vspace*{-6pt}
$$
=-(\nu x^{1-\nu}y^{\nu-1}-1)g(y)\biggr|_{y=x}^{y={+\infty}}+\nu(\nu-1)x^{1-\nu}\int\limits_x^{+\infty} y^{\nu-2}g(y)dy
$$
\vspace*{-6pt}
$$
=(\nu-1)g(x)+\nu(\nu-1)x^{1-\nu}\int\limits_x^{+\infty} y^{\nu-2}g(y)dy.
$$
Integrating by parts the second term once we obtain
$$
\int\limits_x^{+\infty}\left[\left(\frac{x}{y}\right)^{1-\nu}-1\right]g'(y)dy=\left[\left(\frac{x}{y}\right)^{1-\nu}-1\right]g(y)
\biggr|_{y=x}^{y={+\infty}}-\frac{\nu-1}{x^{\nu-1}}\int\limits_x^{+\infty} y^{\nu-2}g(y)dy
$$
\vspace*{-6pt}
$$
=-(\nu-1)x^{1-\nu}\int\limits_x^b y^{\nu-2}g(y)dy.
$$

Hence,
\vskip -10pt
$$
(B_{\nu,-}^{-1}B_\nu g)(x)=\frac{1}{(\nu-1)}\biggl[(\nu-1)g(x)+\nu(\nu-1)x^{1-\nu}\int\limits_x^{+\infty} y^{\nu-2}g(y)dy$$
\vspace*{-6pt}
$$ - \, \nu(\nu-1)x^{1-\nu}\int\limits_x^{+\infty} y^{\nu-2}g(y)dy\biggr]=g(x).
$$
\end{enumerate}
\end{proof} 

\begin{corollary} 
The next four formulas are valid for \, $0<\nu<1$ on proper functions:
$$
(B_{\nu,b-}^{-1}f)(x)=\frac{1}{\nu-1}\int\limits_x^b\, y\,\left[\left(\frac{x}{y}\right)^{1-\nu}-1\right] f(y)dy,
$$
$$
(B_{\nu,a+}^{-1}f)(x)=\frac{1}{\nu-1}\int\limits_a^x y\left[1-\left(\frac{y}{x}\right)^{\nu-1}\right]f(y)dy,
$$
$$
(B_{\nu,-}^{-1}f)(x)=\frac{1}{\nu-1}\int\limits_x^{+\infty}\, y\,\left[\left(\frac{x}{y}\right)^{1-\nu}-1\right] f(y)dy,
$$
$$
(B_{\nu,0+}^{-1}f)(x)=\frac{1}{\nu-1}\int\limits_0^x y\left[1-\left(\frac{y}{x}\right)^{\nu-1}\right]f(y)dy.
$$
\end{corollary}

\section{Mellin transform and fractional powers \break of the Bessel differential operator}

The Mellin transform is a powerful tool for fractional calculus and related differential equations, cf. \cite{KK}, \cite{Mar},
\cite{LuchKir}, \cite{Sita4}.
For more general views and applications of integral transforms method, cf. 
 \cite{Kir1}, \cite{KiSa}.

The Mellin transform of a function $f$ is defined by
\vskip -10pt
 $$
 Mf(s)=f^*(s)=\int\limits_0^\infty x^{s-1}f(x)dx.
 $$

Let us find the result of the Mellin transform applied to the fractional Bessel power operator $B_{\nu,-}^{-\alpha}$.

\begin{theorem} 
The Mellin transform  of $B_{\nu,-}^{-\alpha}$ is determined by the formula
\begin{equation}\label{Mellin}
  ((B_{\nu,-}^{-\alpha}f)(x))^*(s)=\frac{1}{2^{2\alpha}}\,\,\Gamma\left[\begin{array}{cc}
$$\frac{s}{2},$$ & $$\frac{s}{2}-\frac{\nu-1}{2}$$ \\
$$\alpha+\frac{s}{2}-\frac{\nu-1}{2},$$ & $$\alpha+\frac{s}{2}$$  \\
\end{array} \right] f^*(2\alpha+s).
\end{equation}
\end{theorem}

\begin{proof} We start from the definitions
\vskip -10pt
$$
 ((B_{\nu,-}^{-\alpha}f)(x))^*(s)=\int\limits_0^\infty x^{s-1}(B_{\nu,-}^{-\alpha}f)(x)dx$$
 $$=\frac{1}{\Gamma(2\alpha)}\int\limits_0^\infty x^{s-1}dx\int\limits_x^{+\infty}\left(\frac{y^2-x^2}{2y}\right)^{2\alpha-1}
 \,_2F_1\left(\alpha+\frac{\nu-1}{2},\alpha;2\alpha;1-\frac{x^2}{y^2}\right)f(y)dy
 $$
 $$
 =\frac{1}{\Gamma(2\alpha)}\int\limits_0^\infty f(y)(2y)^{1-2\alpha}dy\int\limits_0^y (y^2-x^2)^{2\alpha-1}\,_2F_1\left(\alpha+\frac{\nu-1}{2},\alpha;2\alpha;1-\frac{x^2}{y^2}\right) x^{s-1}dx.
 $$
 Let us find
 \vskip -10pt
 $$
 I=\int\limits_0^y (y^2-x^2)^{2\alpha-1}\,_2F_1\left(\alpha+\frac{\nu-1}{2},\alpha;2\alpha;1-\frac{x^2}{y^2}\right) x^{s-1}dx$$
 \vspace*{-6pt}
 $$=\frac{1}{2}\int\limits_0^{y^2} (y^2-x)^{2\alpha-1}\,_2F_1\left(\alpha+\frac{\nu-1}{2},\alpha;2\alpha;1-\frac{x}{y^2}\right) x^{\frac{s}{2}-1}dx.
 $$
Using the following formula 2.21.1.11 from \cite{IR3}, p. 265,
$$
\int\limits_0^z x^{\alpha-1}(z-x)^{c-1}\,_2F_1\left(a,b;c;1{-}\frac{x}{z}\right)dx{=}z^{c+\alpha-1}\Gamma\left[
                                                           \begin{array}{cc}
                                                             $$c,$$ & $$\alpha,$$ \,\,\,\, $$c{-}a{-}b{+}\alpha$$ \\
                                                             $$c{-}a+\alpha,$$ & $$c{-}b{+}\alpha$$  \\
                                                           \end{array}
                                                         \right],
$$
$$
z>0,\ {\rm Re}\, c>0,\ {\rm Re}\,(c-a-b+\alpha)>0,
$$
we have
$$
z=y^2>0,\ c=2\alpha>0,\ a=\alpha+\frac{\nu-1}{2},\ b=\alpha,\ \alpha=\frac{s}{2},$$
$$c-a-b+\alpha=\frac{s}{2}-\frac{\nu-1}{2}>0\ \Rightarrow \ s>\nu-1.
$$
Then
$$
I=\frac{y^{4\alpha+s-2}}{2}\,\,\Gamma\left[\begin{array}{cc}
$$2\alpha,$$ & $$\frac{s}{2},$$ \,\,\,\,\,\, $$\frac{s}{2}-\frac{\nu-1}{2}$$ \\
$$\alpha+\frac{s}{2}-\frac{\nu-1}{2},$$ & $$\alpha+\frac{s}{2}$$  \\
                                                           \end{array}
                                                         \right]
$$
and
  $$
 ((B_{\nu,-}^{-\alpha}f)(x))^*(s)=\frac{1}{2}\,\,\Gamma\left[\begin{array}{cc}
 $$\frac{s}{2},$$ & $$\frac{s}{2}-\frac{\nu-1}{2}$$ \\
$$\alpha+\frac{s}{2}-\frac{\nu-1}{2},$$ & $$\alpha+\frac{s}{2}$$  \\
                                                           \end{array}
                                                         \right]\int\limits_0^\infty f(y)(2y)^{1-2\alpha}y^{4\alpha+s-2}dy
 $$
$$
=\frac{1}{2^{2\alpha}}\,\,\Gamma\left[\begin{array}{cc}
 $$\frac{s}{2},$$ & $$\frac{s}{2}-\frac{\nu-1}{2}$$ \\
$$\alpha+\frac{s}{2}-\frac{\nu-1}{2},$$ & $$\alpha+\frac{s}{2}$$  \\
                                                           \end{array}
                                                         \right]\int\limits_0^\infty f(y)y^{2\alpha+s-1}dy
$$
$$
=\frac{1}{2^{2\alpha}}\,\,\Gamma\left[\begin{array}{cc}
 $$\frac{s}{2},$$ & $$\frac{s}{2}-\frac{\nu-1}{2}$$ \\
$$\alpha+\frac{s}{2}-\frac{\nu-1}{2},$$ & $$\alpha+\frac{s}{2}$$  \\
                                                           \end{array}
                                                         \right] f^*(2\alpha+s).
$$
So we finally prove that
\begin{equation}\label{Mellin2}
  ((B_{\nu,-}^{-\alpha}f)(x))^*(s)=\frac{1}{2^{2\alpha}}\,\,\Gamma\left[\begin{array}{cc}
$$\frac{s}{2},$$ & $$\frac{s}{2}-\frac{\nu-1}{2}$$ \\
$$\alpha+\frac{s}{2}-\frac{\nu-1}{2},$$ & $$\alpha+\frac{s}{2}$$  \\
                                                           \end{array}
                                                         \right] f^*(2\alpha+s).
\end{equation}
\end{proof} 

It is easy now to check the group property or an index law via the Mellin transform.

\begin{theorem} 
The following group properties or index laws are valid:
\begin{equation}\label{SemiGroup3}
    B_{\nu,-}^{-\alpha}B_{\nu,-}^{-\beta}f= B_{\nu,-}^{-(\alpha+\beta)}f,
\end{equation}
\begin{equation}\label{SemiGroup4}
    B_{\nu,0+}^{-\alpha}B_{\nu,0+}^{-\beta}f= B_{\nu,0+}^{-(\alpha+\beta)}f.
\end{equation}
\end{theorem} 

\begin{proof}
Let $g(y)=(B_{\nu,-}^{-\beta}f)(y)$. Using \eqref{Mellin2} we have
$$
 (B_{\nu,-}^{-\alpha}[(B_{\nu,-}^{-\beta}f)(y)](x))^*(s)=((B_{\nu,-}^{-\alpha}g)(x))^*(s)
 $$
 $$
 =\frac{1}{2^{2\alpha}}\,\,\Gamma\left[\begin{array}{cc}
  $$\frac{s}{2},$$ & $$\frac{s}{2}-\frac{\nu-1}{2}$$ \\
$$\alpha+\frac{s}{2}-\frac{\nu-1}{2},$$ & $$\alpha+\frac{s}{2}$$  \\
                                                           \end{array}
                                                         \right] g^*(2\alpha+s)
 $$
 $$
 = \frac{1}{2^{2\alpha}}\,\,\Gamma\left[\begin{array}{cc}
 $$\frac{s}{2},$$ & $$\frac{s}{2}-\frac{\nu-1}{2}$$ \\
$$\alpha+\frac{s}{2}-\frac{\nu-1}{2},$$ & $$\alpha+\frac{s}{2}$$  \\
                                                           \end{array}
                                                         \right]
 (B_{\nu,-}^{-\beta}f)^*(2\alpha+s)
$$
 $$
 = \frac{1}{2^{2\alpha}}\,\,\Gamma\left[\begin{array}{cc}
 $$\frac{s}{2},$$ & $$\frac{s}{2}-\frac{\nu-1}{2}$$ \\
$$\alpha+\frac{s}{2}-\frac{\nu-1}{2},$$ & $$\alpha+\frac{s}{2}$$  \\
                                                           \end{array}
                                                         \right]$$
$$\times\frac{1}{2^{2\beta}}\,\,\Gamma\left[\begin{array}{cc}
 $$\frac{2\alpha+s}{2},$$ & $$\frac{2\alpha+s}{2}-\frac{\nu-1}{2}$$ \\
$$\beta+\frac{2\alpha+s}{2}-\frac{\nu-1}{2},$$ & $$\beta+\frac{2\alpha+s}{2}$$  \\
                                                           \end{array}
                                                         \right]
f^*(2\alpha+2\beta+s).
$$
On the other hand,
$$
(B_{\nu,-}^{-(\alpha+\beta)}f)^*(s)
=\frac{1}{2^{2(\alpha+\beta)}}\,\,\Gamma\left[\begin{array}{cc}
 $$\frac{s}{2},$$ & $$\frac{s}{2}-\frac{\nu-1}{2}$$ \\
$$\alpha+\beta+\frac{s}{2}-\frac{\nu-1}{2},$$ & $$\alpha+\beta+\frac{s}{2}$$  \\
                                                           \end{array}
                                                         \right] f^*(2\alpha+2\beta+s).
$$
We have the true equality
$$
\frac{1}{2^{2\alpha}}\,\,\Gamma\left[\begin{array}{cc}
 $$\frac{s}{2},$$ & $$\frac{s}{2}-\frac{\nu-1}{2}$$ \\
$$\alpha+\frac{s}{2}-\frac{\nu-1}{2},$$ & $$\alpha+\frac{s}{2}$$  \\
                                                           \end{array}
                                                         \right]\times\frac{1}{2^{2\beta}}\,\,\Gamma\left[\begin{array}{cc}
 $$\frac{2\alpha+s}{2},$$ & $$\frac{2\alpha+s}{2}-\frac{\nu-1}{2}$$ \\
$$\beta+\frac{2\alpha+s}{2}-\frac{\nu-1}{2},$$ & $$\beta+\frac{2\alpha+s}{2}$$  \\
                                                           \end{array}
                                                         \right]
$$
$$
=\frac{1}{2^{2(\alpha+\beta)}}\,\,\Gamma\left[\begin{array}{cc}
 $$\frac{s}{2},$$ & $$\frac{s}{2}-\frac{\nu-1}{2}$$ \\
$$\alpha+\beta+\frac{s}{2}-\frac{\nu-1}{2},$$ & $$\alpha+\beta+\frac{s}{2}$$  \\
                                                           \end{array}
                                                         \right]
$$
that proves \eqref{SemiGroup3}. The relation \eqref{SemiGroup4} is proved similarly.
\end{proof} 

\section{Further properties of fractional powers \break of the Bessel differential operator}

Here let us list some further properties of fractional powers of Bessel differential operator.

\begin{proposition} 
Let $f(x)=x^m$, $x>0, m\in\mathbb{R}$.
Then the integrals $B_{\nu,-}^{-\alpha}$ and $B_{\nu,0+}^{-\alpha}$ of power function are defined by formulas
$$
B_{\nu,-}^{-\alpha}\,x^m=x^{2\alpha+m}\,2^{-2\alpha}\,\Gamma\left[
                                                           \begin{array}{cc}
                                                              $$-\alpha-\frac{m}{2},$$ & $$-\frac{\nu-1}{2}-\alpha-\frac{m}{2}$$ \\
                                                             $$\frac{1-\nu-m}{2},$$ & $$-\frac{m}{2}$$  \\
                                                           \end{array}
                                                         \right],\quad m+2\alpha+\nu<1,
$$
$$
B_{\nu,0+}^{-\alpha}\,x^m=x^{2\alpha+m}\,2^{-2\alpha}\,\Gamma\left[
                                                           \begin{array}{cc}
                                                             $$\frac{m+\nu+1}{2},$$ & $$\frac{m}{2}+1$$  \\
                                                             $$\alpha+\frac{m}{2}+1,$$ & $$\alpha+\frac{m+\nu+1}{2}$$  \\
                                                           \end{array}
                                                         \right].
$$
\end{proposition} 

\begin{corollary} 
The operator $\frac{1}{x^{2\alpha}} B_{\nu,-}^{-\alpha}$ is of the so-called Dzhrbashyan--Gelfond--Leontiev type
(see \cite{KK}, \cite{Kir1}) when $m+2\alpha+\nu<1$. This means that it acts on power series by the rule
$$
\frac{1}{z^{2\alpha}} B_{\nu,-}^{-\alpha}\left(\sum_{k=0}^\infty a_k z^k \right)=
\left(\sum_{k=0}^\infty c(\alpha,k) a_k z^k \right),$$
$$ c(\alpha,k)=2^{-2\alpha}\,\Gamma\left[
                                                           \begin{array}{cc}
                                                              $$-\alpha-\frac{k}{2},$$ & $$-\frac{\nu-1}{2}-\alpha-\frac{k}{2}$$ \\
                                                             $$\frac{1-\nu-k}{2},$$ & $$-\frac{k}{2}$$  \\
                                                           \end{array}
                                                         \right],\quad a_k\in\mathbb{R}.
$$
\end{corollary} 

\begin{corollary} 
The operator $\frac{1}{x^{2\alpha}} B_{\nu,0+}^{-\alpha}$ is of Dzhrbashyan--Gelfond--Leontiev type:
$$
\frac{1}{z^{2\alpha}} B_{\nu,0+}^{-\alpha}\left(\sum_{k=0}^\infty a_k z^k \right)=
\left(\sum_{k=0}^\infty d(\alpha,k) a_k z^k \right),$$
$$ d(\alpha,k)=2^{-2\alpha}\,\Gamma\left[
                                                           \begin{array}{cc}
                                                             $$\frac{m+\nu+1}{2},$$ & $$\frac{m}{2}+1$$  \\
                                                             $$\alpha+\frac{m}{2}+1,$$ & $$\alpha+\frac{m+\nu+1}{2}$$  \\
                                                           \end{array}
                                                         \right],\quad a_k\in\mathbb{R}.
$$
\end{corollary} 

\begin{proposition} 
The formula for integration by parts is valid on proper functions:
\begin{equation}\label{PoChastam}
\int\limits_0^{+\infty}f(x)(B_{\nu,0+}^{-\alpha}g)(x)x^\nu dx=\int\limits_0^{+\infty}g(x)(B_{\nu,-}^{-\alpha}f)(x)x^\nu dx.
\end{equation}
\end{proposition}

\begin{proposition} 
The following formula for the Bessel fractional integral applied to the Bessel function is valid:
\vskip -10pt
\begin{equation}\label{Bess10}
((B_{\nu,0+}^{-\alpha}y^{\frac{1-\nu}{2}}{J}_\frac{\nu-1}{2}(y\xi))(x)=
\frac{1}{x^{\frac{\nu-1}{2}}\xi^{2\alpha}}J_{\frac{\nu-1}{2},\alpha}^1(x\xi).
\end{equation}
Here $J_{\frac{\nu-1}{2},\alpha}^1(\cdot)$ is the Wright function defined by
\vskip -10pt
\begin{equation}\label{GML}
J_{\gamma,\lambda}^\mu(z)=\sum\limits_{m=0}^\infty\frac{(-1)^m}{\Gamma(\gamma+m\mu+\lambda+1)\Gamma(\lambda+m+1)}\left(\frac{z}{2}\right)^{2m+\gamma+2\lambda}.
\end{equation}
\end{proposition} 

Such Wright functions are also often known as generalized Mittag--Leffler functions, cf.
\cite{Kir1}, \cite{Kir2}--\cite{Kir4}, \cite{Rag}, \cite{Jor}, \cite{Vir}.

\section{Hankel transform and fractional powers \break of the Bessel differential operator}

The Hankel transform we apply here is defined as:
\vskip -10pt
\begin{equation}\label{han}
(H_\nu f )(\xi)=\frac{1}{\xi^\nu}\int\limits_{0}^{+\infty}{J}_\frac{\nu-1}{2}(x\xi)\,x^{\frac{\nu+1}{2}}
f(x)\,dx.
\end{equation}

Also in this section we deal with the integral transform of the form
\vskip -10pt
\begin{equation}\label{wright}
(R_\nu^\alpha f)(\xi)=\int\limits_0^{+\infty}J_{\frac{\nu-1}{2},\alpha}^1(x\xi)\,x^{\frac{\nu+1}{2}}f(x)\,dx,
\end{equation}
\vskip -3pt \noindent
where $J_{\frac{\nu-1}{2},\alpha}^1(x)$ is the Wright or generalized Mittag--Leffler function defined in \eqref{GML}.

Now we prove that our version of the explicit fractional powers of the Bessel differential operator defined by integral representations is closely connected with Hankel transform.

\begin{theorem} 
The next formula is valid on proper functions
\begin{equation}
(H_\nu B_{\nu,-}^{-\alpha}f )(\xi)=\xi^{-2\alpha}(R_\nu^\alpha f)(\xi),
\end{equation}
where $R_\nu^\alpha$  is the transform defined by \eqref{wright} with Wright function kernel.
\end{theorem}
\
\begin{proof}
We are using \eqref{PoChastam},
\vskip -10pt
$$
(H_\nu B_{\nu,-}^{-\alpha}f )(\xi)=\int\limits_0^{+\infty}{J}_\frac{\nu-1}{2}(x\xi)\,x^{\frac{\nu+1}{2}}(B_{\nu,-}^{-\alpha}f)(x)\,dx$$
\vspace*{-6pt}
$$=\int\limits_0^{+\infty}((B_{\nu,0+}^{-\alpha}y^{\frac{1-\nu}{2}}{J}_\frac{\nu-1}{2}(x\xi))f(x)\,x^\nu\,dx.
$$
\vskip -3pt \noindent
From \eqref{Bess10} we obtain
\vskip -12pt
$$
(H_\nu B_{\nu,-}^{-\alpha}f )(\xi)=\int\limits_0^{+\infty}((B_{\nu,0+}^{-\alpha}y^{\frac{1-\nu}{2}}{J}_\frac{\nu-1}{2}(x\xi))f(x)\,x^\nu\, dx$$
\vspace*{-6pt}
$$=\xi^{-2\alpha}\int\limits_0^{+\infty}x^{\frac{1-\nu}{2}}J_{\frac{\nu-1}{2},\alpha}^1(x\xi)f(x)\,x^\nu\,dx
$$
\vspace*{-6pt}
$$
=\xi^{-2\alpha}\int\limits_0^{+\infty}J_{\frac{\nu-1}{2},\alpha}^1(x\xi)f(x)\,x^{\frac{\nu+1}{2}}\, dx=\xi^{-2\alpha}(R_\nu^\alpha f)(\xi).
$$
\end{proof} 

\vspace*{-6pt}
\section{Finding a resolvent for  fractional powers \break of the Bessel differential operator}

We consider resolvents for integral operators at standard setting, cf. \cite{KoFo}.
For any linear operator $A$ on some Banach space $\Phi$ let us consider the equation
\begin{equation}\label{EQ00}
\left(A-\lambda I\right)g=f;\qquad \lambda \in C;\qquad f,g\in \Phi,
\end{equation}
and its solution as resolvent operator due to the well--known formula from  \cite{KoFo}
\vskip -10pt
 $$
 g=R_{\lambda } f=\left(A-\lambda I\right)^{-1} f=-\left(\lambda I-A\right)^{-1} f=-\frac{1}{\lambda } \left(I-\frac{1}{\lambda } A\right)^{-1} f
 $$
 \vspace*{-5pt}
\begin{equation}\label{Rez}
 =-\frac{1}{\lambda } \sum _{k=0}^{\infty }\left(\frac{1}{\lambda } A\right) ^{k} f=-\frac{1}{\lambda } f-\frac{1}{\lambda } \left(\sum _{k=1}^{\infty }\frac{A^{k}}{\lambda ^{k} }  f \right).
\end{equation}

Note that if integral representations are known for all powers $A^k$, then an integral representation for the resolvent is readily following from \eqref{EQ00}, of course if the series are convergent. In this way it is possible to get resolvent operators for the Riemann--Liouville fractional integrals, known as the Hille--Tamarkin formula \cite{KK} (in fact first proved by M.M. Dzhrbashyan in \cite{Dzh}),
and also for the Erd\'{e}lyi--Kober fractional integrals but we omit it here.

\begin{theorem} 
For a resolvent  operator of $(B_{\nu,-}^{-\alpha})$ the next formula is valid
\vskip -9pt
$$
 R_{\lambda } f
 =-\frac{1}{\lambda } f-\frac{1}{\lambda^2} \int\limits_x^{+\infty}f(y)\left(\frac{y^2-x^2}{2y}\right)^{2\alpha-1}dy\int\limits_{0}^{1}t^{\alpha-1}(1-t)^{\alpha-1}
  $$
  \vspace*{-4pt}
$$
\times \left(1-\left(1-\frac{x^2}{y^2}\right)t\right)^{-\alpha-\frac{\nu-1}{2}}  E_{(\alpha,\alpha),(\alpha,\alpha)}\left(\frac{1}{\lambda}\left(\frac{1}{4}\frac{t(1-t)(y^2-x^2)^2}
{y^2-(y^2-x^2)t}\right)^{\alpha}\right) dt,
$$
with the Wright or generalized (multi-index) Mittag--Leffler function
\vskip -7pt
\begin{equation}\label{GML2}
    E_{(1/\rho_i),(\mu_i)}(z)=\sum\limits_{k=0}^\infty \frac{z^k}{\Gamma(\mu_1+k/\rho_1)...\Gamma(\mu_m+k/\rho_m)},
\end{equation}
 cf. \cite{Kir1}, \cite{Kir2}--\cite{Kir4}, \cite{Rag}, \cite{Jor}, \cite{Vir}.
\end{theorem} 

\begin{proof}
Let us consider
\vskip -11pt
$$
(B_{\nu,-}^{-\alpha}\,f)(x)=\frac{1}{\Gamma(2\alpha)}\int\limits_x^{+\infty}\left(\frac{y^2-x^2}{2y}\right)^{2\alpha-1}\,_2F_1\left(\alpha+\frac{\nu-1}{2},\alpha;2\alpha;1-\frac{x^2}{y^2}\right)f(y)dy.
$$
\vskip -3pt \noindent
Using the group property or index law, we have
$$
(B_{\nu,-}^{-\alpha}\,f)^k=B_{\nu,-}^{-\alpha k}\,f.
$$
Then  from \eqref{Rez} we obtain
\vskip -9pt
$$
 R_{\lambda } f=-\frac{1}{\lambda } f-\frac{1}{\lambda } \left(\sum _{k=1}^{\infty }\frac{1}{\lambda ^{k} } B_{\nu,-}^{-\alpha k} f \right)=-\frac{1}{\lambda } f--\frac{1}{\lambda } \biggl(\sum _{k=1}^{\infty }\frac{1}{\lambda ^{k}\Gamma(2\alpha k) }
 $$
 \vspace*{-4pt}
$$
\times\int\limits_x^{+\infty}\left(\frac{y^2-x^2}{2y}\right)^{2\alpha k-1}\,_2F_1\left(\alpha k+\frac{\nu-1}{2},\alpha k;2\alpha k;1-\frac{x^2}{y^2}\right)f(y)dy \biggl)
$$
\vspace*{-6pt}
$$
=-\frac{1}{\lambda } f-\frac{1}{\lambda } \left(\int\limits_x^{+\infty}f(y)dy \sum _{k=1}^{\infty }\left[ \frac{1}{\lambda ^{k}\Gamma(2\alpha k) }\left(\frac{y^2-x^2}{2y}\right)^{2\alpha k-1} \right.\right.
$$
\vspace*{-4pt}
$$\left.\left.\times\,_2F_1\left(\alpha k+\frac{\nu-1}{2},\alpha k;2\alpha k;1-\frac{x^2}{y^2}\right)\right]\right).
$$
Using the integral representation for the hypergeometric function for $c - a - b > 0$:
\vskip -10pt
$$
F(a,b;c;z) = { \frac{\Gamma(c)}{\Gamma(b)\Gamma(c-b)} } \int\limits_{0}^{1} t^{b-1} (1-t)^{c-b-1} (1-tz)^{-a} \,dt,
$$
\vskip -4pt \noindent
we obtain
\vskip -13pt
$$
 R_{\lambda } f=-\frac{1}{\lambda } f-\frac{1}{\lambda } \int\limits_x^{+\infty}f(y)dy\int\limits_{0}^{1}  \sum _{k=1}^{\infty } \frac{1}{\lambda ^{k}\Gamma^2(\alpha k) }\left(\frac{y^2-x^2}{2y}\right)^{2\alpha k-1}  t^{\alpha k-1} (1-t)^{\alpha k-1}$$
 \vspace*{-5pt}
 $$\times \left(1-\left(1-\frac{x^2}{y^2}\right)t\right)^{-\alpha k-\frac{\nu-1}{2}}dt
 $$
 \vspace*{-6pt}
 $$
 =\{k=p+1\}=-\frac{1}{\lambda } f-\frac{1}{\lambda } \int\limits_x^{+\infty}f(y)dy\int\limits_{0}^{1}  \sum _{p=0}^{\infty } \frac{1}{\lambda ^{p+1}\Gamma^2(\alpha(p+1)) }\left(\frac{y^2-x^2}{2y}\right)^{2\alpha (p+1)-1}
 $$
 \vspace*{-5pt}
 $$\times  t^{\alpha (p+1)-1} (1-t)^{\alpha (p+1)-1}  \left(1-\left(1-\frac{x^2}{y^2}\right)t\right)^{-\alpha (p+1)-\frac{\nu-1}{2}}dt
 $$
 \vspace*{-6pt}
  $$
 =-\frac{1}{\lambda } f-\frac{1}{\lambda } \int\limits_x^{+\infty}f(y)\left(\frac{y^2-x^2}{2y}\right)^{2\alpha-1}dy
 \int\limits_{0}^{1}t^{\alpha-1}(1-t)^{\alpha-1}\left(1-\left(1-\frac{x^2}{y^2}\right)t\right)^{-\alpha-\frac{\nu-1}{2}}
 $$
 \vspace*{-6pt}
 $$\times\sum_{p=0}^{\infty } \frac{1}{\lambda ^{p+1}\Gamma^2(\alpha(p+1)) }\left(\frac{y^2-x^2}{2y}\right)^{2\alpha p}  t^{\alpha p} (1-t)^{\alpha p}  \left(1-\left(1-\frac{x^2}{y^2}\right)t\right)^{-\alpha p}dt
 $$
 \vspace*{-6pt}
   $$
 =- \, \frac{1}{\lambda } f-\frac{1}{\lambda^2} \int\limits_x^{+\infty}f(y)\left(\frac{y^2-x^2}{2y}\right)^{2\alpha-1}dy
 \int\limits_{0}^{1}t^{\alpha-1}(1-t)^{\alpha-1}\left(1-\left(1-\frac{x^2}{y^2}\right)t\right)^{-\alpha-\frac{\nu-1}{2}}
 $$
 \vspace*{-6pt}
  \begin{equation}\label{Last}
 \times\sum_{p=0}^{\infty } \frac{1}{\Gamma^2(\alpha+\alpha p) }\left[\frac{1}{\lambda}\left(\frac{1}{4}\frac{t(1-t)(y^2-x^2)^2}{y^2-(y^2-x^2)t}\right)^{\alpha}\right]^pdt.
 \end{equation}
 The function in \eqref{Last}   is a special case of the Wright generalized hypergeometric function defined above as \eqref{GML2}.
So it follows
\vskip -10pt
$$
\sum_{p=0}^{\infty } \frac{1}{\Gamma^2(\alpha+\alpha p) }\left[\frac{1}{\lambda}\left(\frac{1}{4}\frac{t(1-t)(y^2-x^2)^2}{y^2-(y^2-x^2)t}\right)^{\alpha}\right]^p$$
\vspace*{-7pt}
$$=E_{(\alpha,\alpha),(\alpha,\alpha)}\left(\frac{1}{\lambda}\left(\frac{1}{4}\frac{t(1-t)(y^2-x^2)^2}
{y^2-(y^2-x^2)t}\right)^{\alpha}\right),
$$
\vskip -3pt \noindent
and we finally derive
\vskip -12pt
$$
 R_{\lambda } f
 =-\frac{1}{\lambda } f-\frac{1}{\lambda^2} \int\limits_x^{+\infty}f(y)\left(\frac{y^2-x^2}{2y}\right)^{2\alpha-1}dy\int\limits_{0}^{1}t^{\alpha-1}(1-t)^{\alpha-1}
  $$
  \vspace*{-6pt}
$$
\times \left(1-\left(1-\frac{x^2}{y^2}\right)t\right)^{-\alpha-\frac{\nu-1}{2}}  E_{(\alpha,\alpha),(\alpha,\alpha)}\left(\frac{1}{\lambda}\left(\frac{1}{4}\frac{t(1-t)(y^2-x^2)^2}{y^2-(y^2-x^2)t}\right)^{\alpha}\right) dt.
$$
\end{proof} 


\section{The generalized Taylor formula with powers of Bessel operators}

Many applications of the Riemann--Lioville fractional integrals are based on the fact that they are remainder terms in Taylor formula. Such formulas exist also with powers of Bessel operators --- they are the so--called Taylor--Delsarte series, cf. \cite{Del1}--\cite{Lev3} and especially \cite{Fag1}. But in the Taylor--Delsarte series not a function itself is expanded but its generalized translation,
these series are in fact just operator versions of Bessel function series. But for application to numerical PDE solution we need the classical form of the Taylor formula $f(x+t)=f(x)+\dots$ , only with it we may calculate  PDE solutions layer by layer.
Such formulas are much harder to prove. With the above mentioned motivation as a tool for solving singular PDE numerically a first attempt to construct the generalized Taylor formula with Bessel operators was done in \cite{Kat1}--\cite{Kat2}.
But these results were rather vague as neither coefficients nor integral remainder term were found explicitly: for coefficients the recurrent system of equations was found and for the remainder term its evaluation as a multi--term composition of simple integral operators.
The solution to this problem of finding the generalized Taylor formula with Bessel operators in the explicit form
was found in \cite{Sita3}, also cf. \cite{Sita1}--\cite{Sita2}, \cite{Sita4}. Of course it is based on explicit forms for fractional powers of Bessel operators.

\begin{theorem}  
The next generalized Taylor formula is valid for proper functions:
\begin{eqnarray}
f(x)=\sum_{i=1}^{k}
\Biggl\{
 \frac{1}{\Gamma(2i-1)}
\Bigl(\frac{b^{2}-x^{2}}{2b}\Bigr)^{2i-2}
{_2}F_{1}\left(i+\frac{{\nu}-1}{2},i-1;2i-1;1-\frac{x^{2}}{b^{2}}\right) \nonumber\\
\times\,(B^{i-1}f)|_{b}- \frac{1}{\Gamma(2i)}
\Bigl(\frac{b^{2}-x^{2}}{2b}\Bigr)^{2i-1} {_2}F_{1}\left(i+\frac{{\nu}-1}{2},i;2i;1-\frac{x^{2}}{b^{2}}\right) \\
\nonumber \times\, (DB^{i-1}f)|_{b}
\Biggr\} +B_{{\nu},b-}^{-k}(B^{k}f).\phantom{11111111111111111}
\end{eqnarray}
\end{theorem} 

The proofs of these two theorems are direct but rather tedious.
Similar results are valid for the (Clifford-type Bessel) differential operator
\vskip -11pt
\begin{equation*}
C_{{\nu}}f=D^{2}f-D\left(\frac{{\nu}}{y}f\right)=D^{2}-\frac{{\nu}}{y}Df+\frac{{\nu}}{y^{2}}.
\end{equation*}

\begin{theorem} 
The next generalized Taylor formula is valid for proper functions:
\begin{eqnarray}
f(x)=\sum_{i=1}^{k}
\Biggl\{\frac{1}{\Gamma(2i-1)}
\Bigl(\frac{x^{2}-a^{2}}{2x}\Bigr)^{2i-2} (\frac{a}{x})\  {_2}F_{1}
\left(i+\frac{{\nu}-1}{2},i;2i-1;1-\frac{a^{2}}{x^{2}}\right)\nonumber\\
\times(C_{{\nu}}^{i-1}f)|_{a}
+\frac{1}{\Gamma(2i)}
\Bigl(\frac{x^{2}-a^{2}}{2x}\Bigr)^{2i-1} \phantom{1111111}\\
\nonumber \times  \ {_2}F_{1}\left(i+\frac{{\nu}-1}{2},i;2i; 1-\frac{a^{2}}{x^{2}}\right) \ a^{{\nu}} \ (Dx^{-{\nu}}C_{{\nu}}^{i-1}f)|_{a}
\Biggr\} +B_{{\nu},a+}^{-k}(C_{{\nu}}^{k}f).
\end{eqnarray}
\end{theorem}  

All the hypergeometric functions in the above formulas may be expressed via Legendre functions as in Definitions 3 and 4.

\vspace*{-5pt}
\section{Generalizations and comments}

1. As we know the generalized Taylor series, then the definitions of the Bessel fractional integrals may be modified as Bessel--Gerasimov--Caputo fractional integrals \cite{KK} in obvious way by subtracting sections of series.

2. As in the paper \cite{Ida}, the more general fractional powers of operators $(\dfrac{d}{xdx})^\beta B_\nu^\alpha$ may be studied.

3. In \cite{GaPo} fractional powers of Bessel operators are applied in some cases for solving the fractional Klein--Gordon equation.
For such problems the most important case is $\nu=1$, for which the integral representations from Definitions 1--4 are simplified to kernels with classical Legendre functions $P_\nu(z)$, cf. \cite{IR3}.

4. In case of powers  $B_\nu^{-1/2}$ for all considered cases the kernels are reduced to generalized elliptic functions.
So generalized  Riesz transforms for the Bessel differential operator may be defined
explicitly, similarly as it is done in \cite{Vil} for the Laplace operator restricted to radial functions.

5.  Let $\nu_i>0$, $i=1,...,n$ and $B_{\nu_i}$ being the operator \eqref{Bess}.
Negative real powers of $\sum\limits_{i=1}^nB_{\nu_i}$ known as B--Riesz potentials for elliptic case and $\sum\limits_{i=1}^pB_{\nu_i} \!-\! \sum\limits_{i=p+1}^nB_{\nu_i}$  for hyperbolic and ultrahyperbolic cases were studied in \cite{Shi0}-\cite{Shi2}.
Using fractional powers of the Bessel differential operator defined and studied in this paper, we may introduce negative real powers for more general fractional differential operators of the form
\vskip -14pt
$$
\sum\limits_{i=1}^n B_{\nu_i}^{-\alpha_{i}}, \ \ \sum\limits_{i=1}^p B_{\nu_i}^{-\alpha_{i}}-\sum\limits_{i=p+1}^n B_{\nu_i}^{-\alpha_{i}}
$$
for all cases of Bessel fractional operators. Note that the B--Riesz potentials and so their further generalizations are closely connected with mean formulas and generalized spherical means for Bessel--type PDEs, cf. \cite{Shi3}--\cite{Shi5}.

6. Let $B_{\nu_i}^{-\alpha}$ be any of fractional Bessel operators from Definitions 1--4.
PDE equations with such operators (FrPDE) may be considered instead of varied Riemann-Liouville fractional operators for known types of fractional differential equations (FrDE), cf. \cite{Kil2}.

7. Fractional Bessel operators may be used in the transmutation theory \cite{Dim},
\cite{Kir1}, \cite{Car1}--\cite{Kra},  \cite{Sita4}.
In this way new classes of transmutations may be introduced, cf. \cite{Sita5}--\cite{Sita7}, and new singular PDEs considered by transmutation methods, cf. \cite{Sita8}--\cite{Sita9}.

8. The fractional powers of the Bessel operator are also important for the theory of commuting differential operators.
In fact the operator \eqref{Bess} differs only by power substitute from the angular momentum operator in quantum mechanics $D^2 - \frac{l(l+1)}{x^2}$. Finding differential operators commuting with it for $l=1$ use the Burchnall--Chaundy theory,
but it is interesting that for other integer $l$ commuting operators are derived using combination of this theory and
fractional powers of the Bessel operator $(B_{\nu})^{\frac{1}{2}}$ \cite{Mir}.

\vspace*{-0.4cm} 

\bigskip


\begin{thebibliography}{00} 

\bibitem{SiShi} E.L.Shishkina, S.M. Sitnik. \emph{On fractional powers of Bessel operators.}
Journal of Inequalities and Special Functions. (Special issue To honor Prof. Ivan Dimovski's contributions).
2017, Vol. 8,  Issue 1, P. 49--67.

\bibitem{CSh} R. W. Carroll, \textit{Singular and Degenerate Cauchy problems},   Academic Press, New York (1976).

\bibitem{Kip} I. A. Kipriyanov, \textit{ Singular Elliptic Boundary Value Problems},  Nauka, Moscow (1997) (in Russian).

\bibitem{Mat} M. I. Matiychuk, \textit{Parabolic Singular Boundary-Value Problems}, Kiev (1999) (in Ukrainian).

\bibitem{McB} A. C. McBride, \textit{ Fractional Calculus and Integral Transforms of Generalized Functions}, Pitman, London (1979).

\bibitem{KK} S. G. Samko, A. A. Kilbas  and O. I. Marichev,  \textit{Fractional Integrals and Derivatives, Theory and Applications}, Gordon and Breach Sc. Publ., Amsterdam (1993).

\bibitem{Ida} I. G. Sprinkhuizen-Kuyper,  \textit{A fractional integral operator corresponding to negative powers of a certain second-order differential operator},  J. Math. Analysis and Applications {\bf 72} (1979) 674--702.

\bibitem{Dim66}
I. Dimovski, \emph{Operational calculus for a class of differential operators},
C. R. Acad. Bulg. Sci. \textbf{19}, 12 (1966) 1111--1114.

\bibitem{Dim68}
I. Dimovski, \emph{On an operational calculus for a differential operator},
C. R. Acad. Bulg. Sci. \textbf{21},  6 (1968) 513--516.

\bibitem{Dim} I. Dimovski,  \textit{Convolutional Calculus}, Kluwer, Dordrecht (1990). %

\bibitem{DimKir}
I. H. Dimovski, V. S. Kiryakova,
\emph{Transmutations, convolutions and fractional powers of Bessel-type operators via Meijer's $G$-function},
in: ``Complex Analysis and Applications '83'' (Proc. Intern. Conf. Varna 1983), Sofia (1985) 45--66.

\bibitem{Kir1} V. Kiryakova,  \textit{Generalized Fractional Calculus and Applications}, Pitman Res. Notes Math.
301, Longman Scientific \& Technical, Harlow, Co-publ. John Wiley, New York (1994).

\bibitem{Sita1} S. M. Sitnik,
\textit{On explicit definitions of fractional powers of the Bessel differential operator and its applications to differential equations},
Reports of the Adyghe (Circassian) International Academy of Sciences  {\bf 12}, 2 (2010)  69--75 (in Russian).

\bibitem{Sita2} S. M. Sitnik,  \textit{Fractional integrodifferentiations for differential Bessel operator}, in:
Proc. of the International Symposium ``The Equations of Mixed Type and Related Problems of the Analysis and Informatics",
Nalchik (2004) 163--167 (in Russian).

\bibitem{Sita3} D. S. Konovalova,  S. M. Sitnik, \textit{The Taylor formula for the  Bessel--type operators},  in:
Voronezh  School ``Modern Mathematical Methods in the Theory of Boundary Value Problems. Pontryagin Reading VII", Abstracts, Voronezh (1996) 102 (in Russian).

\bibitem{Sita4} S. M. Sitnik,  \textit{ Transmutations and applications: A survey}, arXiv:1012.3741v1, 141 pp.

\bibitem{IR3} A. P. Prudnikov, Yu. A. Brychkov, O. I. Marichev,  \textit{Integrals and Series, Vol. 3, More Special Functions},
Gordon \& Breach Sci. Publ., New York (1990).

\bibitem{Mar} O. I. Marichev, \textit{Method for Computing Integrals of Special Functions},  Minsk (1978) (in Russian).

\bibitem{LuchKir}
Yu. Luchko, V. Kiryakova, \emph{The Mellin integral transform in fractional calculus},
Fract. Calc. Appl. Anal. \textbf{16}, 2 (2013) 405--430.

\bibitem{KiSa} A. A. Kilbas,  M. Saigo \textit{H--Transforms, Theory and Applications}, Chapman \& Hall/CRC, Boca
Raton-London-New York (2004).

\bibitem{Kir2} V. S. Kiryakova,  \textit{Multiple (multi-index) Mittag--Leffler functions and relations to generalized
fractional calculus}. J. of Computational and Applied Mathematics  {\bf 118}, 1-2 (2000) 241-–259.

\bibitem{Kir3} V. S. Kiryakova,  \textit{The special functions of fractional calculus as generalized fractional
calculus operators of some basic functions}, Computers and Mathematics with Applications  {\bf 59}, 3 (2010) 1128--1141.

\bibitem{Kir4} V. S. Kiryakova,  \textit{The multi-index Mittag-Leffler functions as important class of special
 functions of fractional calculus}, Computers and Mathematics with Applications {\bf 59}, 5 (2010) 1885--1895.

\bibitem{Rag} R. Gorenflo, A. A. Kilbas, F. Mainardi, S. V. Rogosin,  \textit{Mittag--Leffler Functions, Related Topics and Applications}, Springer (2014).

\bibitem{Jor} J. Paneva--Konovska, \textit{From Bessel to Multi--Index Mittag--Leffler Functions}, World Scientific, London (2016).

\bibitem{Vir} N. Virchenko, V. Gaidei, \textit{Classic and Generalized Multi--Parameters Function}, Kiev (2008) (in Ukranian).

\bibitem{KoFo} A. N. Kolmogorov, S. V. Fomin, \textit{Introductory Real Analysis} (1970).

\bibitem{Dzh} M. M. Dzhrbashyan, \textit{ Integral Transforms and Representations  of  Functions in Complex Plane},  Moscow, Nauka (1966) (in Russian). 

\bibitem{Del1} J. Delsarte, \textit{Sur une extension de la formule de Taylor}, Journ. Math. pures et appl. {\bf 17} (1938) 217--230.

\bibitem{Lev1} B. M. Levitan, \textit{Expansion in Fourier series and integrals with Bessel functions}, Uspekhi Mat. Nauk {\bf 6}, 2(42) (1951) 102–-143 (in Russian).

\bibitem{Lev2} B. M. Levitan, \textit{Generalized Translation Operators and Some Their Applications}, Moscow (1962) (in Russian).

\bibitem{Lev3} B. M. Levitan, \textit{Generalized Translation Operators Theory}, Moscow (1973) (in Russian).

\bibitem{Fag1} D. K. Fage, N. I. Nagnibida, \textit{The Problem of Equivalence of Ordinary Differential Operators}, Novosibirsk (1977) (in Russian).

\bibitem{Kat1} V. V. Katrahov, A. A. Katrahova, \textit{On the approximation of solutions of some singular elliptic problems}, DAN SSSR {\bf 249}, 1 (1979) 34--37 (in Russian).

\bibitem{Kat2} V. V. Katrahov, A. A. Katrahova, \textit{Taylor formula with Bessel operator for one and two variable function}, Dep. VINITI, Voronezh (1982) (in Russian).

\bibitem{GaPo} R. Garra, E. Orsingher, F. Polito, \textit{Fractional Klein--Gordon equation for linear
dispersive phenomena: Analytical methods and applications}, IEEE Xplore (2014).

\bibitem{Vil} M. Villani, \textit{Riesz transforms associated to Bessel operators}, Illinois Journal of Mathematics  {\bf 52}, 1 (2008) 77--89.

\bibitem{Shi0} L. N. Lyakhov,  E. L. Shishkina,
\textit{Inversion of general Riesz B-potentials with homogeneous characteristic in weight classes of functions},
Doklady Akademii Nauk {\bf 426}, 4 (2009) 443-–447 (in Russian).

\bibitem{Shi1} E. L. Shishkina, \textit{On the boundedness of hyperbolic Riesz B-potential},
Lithuanian Mathematical Journal {\bf 56}, 4 (2016)  540-–551.

\bibitem{Shi2} E. L. Shishkina,  \textit{Inversion of integral of B-potential type with density from $\Phi_\gamma$},
Journal of Mathematical Sciences  {\bf 160}, 1 (2009)  95-–102.

\bibitem{Shi3} L. N. Lyakhov,  E. L. Shishkina, \textit{ Weighted mixed spherical means and singular ultrahyperbolic equation},
Analysis (De Gruyter), International Mathematical Journal of Analysis and its Applications {\bf 36}, 2 (2016) 65--70.

\bibitem{Shi4} E. L. Shishkina, S. M. Sitnik,
\textit{ On an identity for the iterated weighted spherical mean and its applications},
Siberian Electronic Mathematical Reports {\bf 13} (2016) 849--860.

\bibitem{Shi5} L. N. Lyakhov, I. P. Polovinkin, E. L. Shishkina, \textit{ On Kipriyanov problem for a singular
ultrahyperbolic equation}, Differential Equations (Springer) {\bf 50}, 4 (2014) 513--525.

\bibitem{Kil2} A. A. Kilbas, H. M. Srivastava, J. J. Trujillo,
\textit{Theory and Applications of Fractional Differential Equations}, Elsevier (2006).

\bibitem{Car1} R. W. Carroll, \textit{Transmutation and Operator Differential Equations}, Mathematics
Studies, Vol. 37, North Holland (1979).

\bibitem{Car2} R. W. Carroll, \textit{Transmutation, Scattering Theory and Special Functions}, Mathematics
Studies, Vol. 69, North Holland (1982).

\bibitem{Car3} R. W. Carroll, \textit{Transmutation Theory and Applications}, Mathematics Studies, Vol. 117,
 North Holland (1985).

\bibitem{BeGi}  H. Begehr, R. P. Gilbert, \textit{Transformations, Transmutations and Kernel Functions},
Vols. 1, 2, Pitman Monographs and Surveys in Pure and Applied Mathematics, Vol. 59,
Harlow, Longman Scientific \& Technical (1993).

\bibitem{Tri} Kh. Trim\'eche, \textit{Transmutation Operators and Mean–-Periodic Functions Associated with
Differential Operators}, Mathematical Reports, Vol. 4, Part 1, Harwood Academic Publishers (1988).

\bibitem{Kra} V. V. Kravchenko,  \textit{Applied Pseudoanalytic Function Theory}, Birkh\"auser, Basel (2009).

\bibitem{Sita5} S.M. Sitnik, \textit{Buschman-Erd\'{e}lyi transmutations, classification and applications},
in: Analytic Methods of Analysis and Differential Equations: AMADE 2012 (Ed. by M.V. Dubatovskaya,  S.V. Rogosin),
Cambridge Sci. Publishers, Cambridge (2013) 171--201.

\bibitem{Sita6} V. V. Katrakhov,  S. M. Sitnik,  \textit{Composition method for constructing B--elliptic,
B--hyperbolic, and B--parabolic transformation operators},  Doklady Mathematics  {\bf 50}, 1 (1995) 70--77 (in Russian).

\bibitem{Sita7} S.M. Sitnik, \textit{Factorization and estimates of the norms of Buschman--Erdelyi operators
in weighted Lebesgue spaces}, Doklady Mathematics  {\bf 44}, 2 (1992) 641--646.

\bibitem{Sita8} V. V. Katrakhov, S. M. Sitnik, \textit{A boundary--value problem for the steady--state
Schr\"odinger equation with a singular potential}, Soviet Mathematics. Doklady {\bf 50}, 1 (1995) 468-–470.

\bibitem{Sita9}  V. V. Katrakhov, S. M. Sitnik, \textit{Estimates of the Jost solution to a one--dimensional
Schr\"odinger equation with a singular potential}, Soviet Mathematics. Doklady {\bf 51}, 1 (1995) 14--16.

\bibitem{Mir} A. E. Mironov, \textit{Self--adjoint commuting differential operators of rank two},
Russian Mathematical Surveys  {\bf 71}, 4 (2016) 751--779.

\end{thebibliography}
\end{document}